\documentclass[a4paper,reqno,oneside,11pt]{amsart} 
\reversemarginpar
\usepackage[utf8]{inputenc}
\usepackage[english]{babel}
\usepackage{hyperref}
\usepackage{amsmath}
\usepackage{amsthm}
\usepackage{amsfonts}
\usepackage{amssymb}
\usepackage{enumerate}
\usepackage{mathrsfs}
\usepackage{xcolor}
\usepackage{latexsym}
\usepackage{dsfont}
\usepackage{bbm}

\numberwithin{equation}{section}
\allowdisplaybreaks

\theoremstyle{plain}
\newtheorem{theorem}{Theorem}[section]
\newtheorem*{thmA}{Theorem A}
\newtheorem{proposition}[theorem]{Proposition}
\newtheorem{lemma}[theorem]{Lemma}
\newtheorem{corollary}[theorem]{Corollary}

\theoremstyle{definition}

\newtheorem{remark}[theorem]{Remark}

\newcommand{\N}{\mathbb{N}}
\newcommand{\Z}{\mathbb{Z}}

\newcommand{\C}{\mathbb{C}}
\newcommand{\T}{\mathbb{T}}
\newcommand{\D}{\mathbb{D}}
\newcommand{\A}{\mathbb{A}}
\newcommand{\B}{\mathbb{B}}

\newcommand{\al}{\alpha}
\newcommand{\be}{\beta}

\newcommand{\om}{\omega}

\newcommand{\si}{\sigma}

\newcommand{\Om}{\Omega}

\renewcommand{\phi}{\varphi}

\newcommand{\w}{\omega}

\newcommand{\CO}{\mathcal{O}}

\newcommand{\BC}{\mathbf{C}}

\newcommand{\FD}{\mathfrak{D}}

\newcommand{\cl}[1]{{#1}^\text{\rm cl}}

\newcommand{\sm}{\setminus}
\renewcommand{\ss}{\subset}

\newcommand{\norm}[1]{\lVert #1 \rVert}

\newcommand{\wt}[1]{\widetilde{#1}}

\newcommand{\abs}[1]{\lvert #1 \rvert}

\newcommand{\Sp}{\operatorname{Sp}}

\newcommand{\ip}[2]{\left\langle #1,#2 \right\rangle}

\renewcommand{\H}{H}

\newcommand{\E}{E}

\makeatletter
\@namedef{subjclassname@2020}{%
  \textup{2020} Mathematics Subject Classification}
\makeatother

\begin{document}
\title[An operator model in the annulus]{An operator model in the annulus}
\author[G. Bello]{Glenier Bello}
\address{G. Bello \newline
Institute of Mathematics of the Polish Academy of Sciences\\
00-656 Warszawa\\
ul. \'{S}niadeckich 8\\
Poland}
\email{gbello@impan.pl}

\author[D. V. Yakubovich]{Dmitry Yakubovich}
\address{D. V. Yakubovich \newline
Departamento de Matem\'aticas\\
Universidad Aut\'onoma de Madrid\\
Cantoblanco, 28049 Madrid, Spain\\
and Instituto de Ciencias Matem\'aticas (CSIC-UAM-UC3M-UCM)}
\email {dmitry.yakubovich@uam.es}
\date{\today}
\begin{abstract}
For an invertible linear operator $T$ on a Hilbert space $\H$, put 
\[
\al(T^*,T) := -T^{*2}T^2 + (1+r^2) T^* T - r^2 I, 
\]
where $I$ stands for the identity operator on $\H$ and $r\in (0,1)$; this expression 
comes from applying Agler's hereditary functional calculus to 
the polynomial $\al(t)=(1-t) (t-r^2)$. 
We give a concrete unitarily equivalent functional model for operators 
satisfying $\al(T^*,T)\ge0$. In particular, we prove that 
the closed annulus $r\le |z|\le 1$ is a complete $K$-spectral set for $T$. 
We explain the relation of the model with the Sz.-Nagy--Foias one and with  
the observability gramian and discuss the relationship of this class with other operator classes related to the annulus. 
\end{abstract}
\maketitle
\section{Introduction}
Given a bounded subset $\Om$ of the complex plane, and a Hilbert space operator $T$ 
with spectrum in $\Om$, the closure $\cl\Om$ of $\Om$ is said to be a 
$K$-spectral set for $T$, for some constant $K\ge 1$, if 
\[
\norm{f(T)}\le K\max_{z\in\cl\Om}\norm{f(z)}
\]
for any rational function $f$ with poles outside $\cl\Om$. 
The notion of spectral sets (i.e.,
$K$-spectral for $K=1$) was introduced by von Neumann in \cite{vNe51}. 
If the same inequality holds for any rational $n\times n$ matrix-valued function $f$ with 
poles outside $\cl\Om$, for any size $n$, then $\cl\Om$ is said to be completely $K$-spectral 
for $T$. 
Arveson \cite{Arv69} proved that $\cl\Om$ is a complete $K$-spectral set 
for $T$ for some $K\ge1$ if and only if $T$ is similar to an operator 
which has a normal dilation with spectrum contained in the boundary of $\Om$. 
The case when $\Om$ is the unit disc $\D$ is deeply related with the 
Sz.-Nagy--Foias theory of Hilbert space contractions. 
An excellent reference for this theory is the book \cite{NFBK10}. 
In the landmark paper \cite{Agler85}, Agler studied the case when $\Om$ is 
an annulus 
\[
\A:=\{r<\abs{z}<1\},
\] 
for some $0<r<1$. He proved that 
$\cl\A$ is $1$-spectral for $T$ if and only if $\cl\A$ is completely $1$-spectral for $T$. 
Moreover, it is well-known that operators for which $\cl\A$ is a $1$-spectral set, 
that we will denote by $\Sp\A$, admit the following model: 

\begin{thmA}[cf.\ {\cite[Theorem~2.1]{Ball79lifting}}]\label{thm-1-spectral}
The set $\cl\A$ is $1$-spectral for $T$ if and only if 
there exist unitary operators $U_1$, $U_2$ and a weight $\w$ on a Hilbert space $\E$ 
such that $T$ is unitarily equivalent to a compression of the operator 
\[
(M_z \text{ on }H^2(\A,\E, \w))\oplus U_1\oplus rU_2
\] 
to its coinvariant subspace. 	
\end{thmA}

Recall that, given a Hilbert space $\E$, the vector valued Hardy space $H^2(\A,\E)$ consists of 
all analytic functions $f\colon\A\to\E$ such that 
\[
\sup_{r<\rho<1}\int_\T\norm{f(\rho z)}^2\,\abs{dz}<\infty. 
\]
There is no unique canonic way of choosing a Hilbert space norm on this space.  
For any bounded, positive and invertible operator $\om$ on $\E$,   
$H^2(\A,\E,\om)$ is defined as the vector space $H^2(\A,\E)$ 
equipped with the Hilbert space norm
\begin{equation}
\label{eq:norm-omega}
\norm{f}_{H^2(\A,\E,\om)}^2:=\int_\T\norm{f(z)}_E^2\,\abs{dz}+
\int_{r\T}\norm{\om f(z)}_E^2\,\abs{dz}
\end{equation}
All these norms are equivalent. 
The operator $M_z$ acts on 
$H^2(\A,\E,\om)$ by $M_zf(z):=zf(z)$. These operators 
are purely subnormal and mutually similar, 
but, in general, not unitarily equivalent one to another.  

For a general finitely connected domain $\Om$, Abrahamse and Douglas 
\cite{AbrDouglas1976} showed the importance of what they called bundle shifts; 
they form a subclass of pure subnormal operators with spectrum in $\cl\Om$. 
In the case of an annulus, the set of bundle shifts coincides with 
the family of operators $M_z$ on spaces $H^2(\A,\E,\om)$. 

Consider the polynomial $\al(t):=(1-t)(t-r^2)$. 
For any Hilbert space operator $T$, we define 
\begin{equation}\label{eq:alphaTT}
\al(T^*,T):=-T^{*2}T^2 + (1+r^2) T^* T - r^2 I, 
\end{equation}
where $I$ stands for the identity operator on the space where $T$ acts. 
Here we put all expressions $T^*$ on the left, 
because we are applying the so-called Agler's hereditary functional calculus, 
whose relevance in constructing operator models is well-known.  
We denote by $\BC_\al$  the class of all invertible bounded linear operators $T$ 
such that the operator $\al(T^*,T)$ is non-negative. 
Let $\BC_{1,r}$ stand for the family of all bounded invertible operators $T$ 
such that $\norm{T}\le1$ and $\norm{T^{-1}}\le1/r$. 
We have the following strict inclusions of classes of operators.

\begin{theorem}\label{thm:strictInclusions}
$\Sp\A\subsetneq\BC_\al\subsetneq\BC_{1,r}$.
\end{theorem}

In this paper we focus on the study  of the class $\BC_\al$. 
It is natural to try an approach in the spirit of Sz.-Nagy--Foias theory. 
Suppose we are given an operator $T$ in $\BC_\al$. 
Denote by $D$ the non-negative square root of $\al(T^*,T)$ 
and by $\FD$ the closure of the range of $D$; 
these will be called the defect operator and the defect space of $T$, respectively. 
We will construct an explicit model 
where $D$ plays the role of the abstract defect operator. More precisely,  
the lifting of the model will involve the output transform 
\[
\CO_{T,D}\colon H\to H^2_r(\B,\FD), \quad 
\CO_{T,D}x(z)=D(z-T)^{-1}x, \quad x\in H, z\in\B. 
\]
Here $\B$ denotes the complementary of $\cl\A$. 
For any Hilbert space $\E$, 
the space $H^2_r(\B,\E)$ consists of all functions $f\colon\B\to\E$ representable as 
\[
\sum_{n=0}^{\infty} f_nz^n \ \text{ for } \lvert z\rvert<r, 
\qquad
\sum_{n=-\infty}^{-1} f_nz^n \ \text{ for } \lvert z\rvert>1,
\]
where $\{f_n\}_{n\in\Z}$ is a sequence in $\E$, with finite norm
\[
\norm{f}_{H^2_r(\B,\E)}^2 := \frac{1}{1-r^2} \sum_{n=0}^{\infty} r^{2n}\norm{f_n}^2 
+ \frac{1}{1-r^2}\sum_{n=-\infty}^{-1} \norm{f_n}^2.  
\]
On these spaces the operator $M_z^t$ acting by 
\begin{equation}\label{eq:Mzt}
M_z^tf(z)=z f(z)-\big(z f(z)\big)\vert_{z=\infty}
\end{equation}
is well-defined. 
If we identify a function $f$ in $H^2_r(\B,\E)$ with the two-sided vector sequence 
$(\ldots f_{-3},f_{-2},f_{-1}, f_0, f_1, f_2, \ldots)$,
then $M_z^t$ takes the form of the bilateral shift. 
Our model theorem for operators in $\BC_\al$ is the following. 

\begin{theorem}[Model theorem]\label{thm:main-b}
Suppose that $T$ is invertible.  
The following statements are equivalent. 
\begin{enumerate}[\rm(i)]
\item 
$\al(T^*,T) \ge 0$.
\item 
$T$ is unitarily equivalent to a part of an operator of the form 
\[
(M_z^t \text{ on } H^2_{r}(\B,\E))\oplus S,
\] 
where $E$ is a Hilbert space 
and $S$ is a subnormal operator whose minimal unitary extension
has spectrum contained in the union of the circles 
$\{\abs{z}=r\}$ and $\{\abs{z}=1\}$. 
\end{enumerate}
If {\rm(i)} holds, one can take $\E=\FD$ in {\rm(ii)}. 
\end{theorem}

Using a certain duality between models, in Theorem~\ref{thm-1-2} 
we obtain a model for $T^*$ with the structure of Theorem~A. 
The explicit model permits to obtain a concrete value of $K$ such that 
$\cl\A$ is completely $K$-spectral for operators in $\BC_\al$ 
(see Theorem~\ref{thm:sqrt2}). 

Let us give a few comments on 
the literature concerning the annulus. 
Apart from Agler's paper cited above, another 
classic work is that by Sarason \cite{Sar65}.  
He obtained for $H^p(\A)$ analogous results to the 
well-established theory for $H^p(\D).$ 
The transition is not always smooth; for example, 
Blaschke products cannot be implemented in the annulus. 
Sarason overcame this obstacle by introducing what he called 
\emph{modulus automorphic functions}. 
In the second part of the paper, he studied invariant and doubly invariant subspaces 
for the multiplication operator on $L^2(\partial\A).$ 
Recent papers dealing with the annulus are for example \cite{McCS12,Pie19,Pie20} 
(see also the references therein). 
In \cite{McCS12}, McCullough and Sultanic proved a kind of 
commutant lifting theorem for the annulus. 
Earlier, in a different context, a result on commutant lifting for finitely connected domains 
has been obtained by Ball in \cite{Ball79lifting}. 
In \cite{Pie19}, Pietrzycki obtained an analytic model on an annulus for left-invertible operators. 
This model allowed him to extend in \cite{Pie20} the notion of 
generalized multipliers for left-invertible analytic operators, 
introduced in \cite{DPP19}, to left-invertible operators. 
These works were motivated by weighted shifts on directed trees.  

Our forthcoming work \cite{BYdoms} is devoted to operator theory corresponding to 
general multiply connected domains, in a somewhat different context. 
Some consequences of the results of the present paper will be derived there. 

The paper is organized as follows. 
In Section~\ref{sec:model} we obtain basic properties  of operators in $\BC_\al$ 
and prove our model theorem. 
In Section~\ref{sec:duals} we study the duality inherent to this work and prove 
Theorem~\ref{thm:strictInclusions}. 
Finally, in Section~\ref{sec:consequences} we derive consequences from 
the explicit model obtained.

\section{An explicit functional model for operators in $\BC_\al$}
\label{sec:model}

This section is devoted to proving Theorem~\ref{thm:main-b}. 

Fix a bounded linear operator $T$ acting on a Hilbert space $H.$ 
Notice that $\al(T^*,T)\ge0$ if and only if 
\begin{equation}
\label{eq:normDx}
(1+r^2)\norm{Tx}^2-\norm{T^2x}^2-r^2\norm{x}^2\ge0
\end{equation}
for all $x$ in $\H$. 
Observe that the left hand side of \eqref{eq:normDx} is precisely $\norm{Dx}^2$. 
In fact, for all $x\in\H$ and all $n\in\Z$ we have
\begin{equation}\label{eq:iteration}
\norm{DT^nx}^2=(1+r^2)\norm{T^{n+1}x}^2-\norm{T^{n+2}x}^2-r^2\norm{T^{n}x}^2\ge0 . 
\end{equation}

Simple computations using \eqref{eq:normDx} reveal that the membership of 
operators $T$ and $rT^{-1}$ in the class $\BC_\al$ are closely related. 
More precisely: 

\begin{proposition}
\label{prop:bothCalpha}
Suppose that $T$ is invertible. Then $T$ is in $\BC_\al$ if and only if $rT^{-1}$ is in $\BC_\al$. 
\end{proposition}

We next obtain the inclusion $\BC_\al\ss\BC_{1,r},$ 
stated in Theorem~\ref{thm:strictInclusions}. 
The complete proof of that theorem will be given at the end of Section~\ref{sec:duals}. 

\begin{proposition}\label{prop:Tcontraction}
If $T$ is in $\BC_\al$, then $\norm{T}\le 1$ and $\norm{T^{-1}}\le 1/r$.  
\end{proposition}

\begin{proof}
Note that the inequality in \eqref{eq:iteration} can be written as 
\[
\frac{\norm{T^{n+1}x}^2-\norm{T^nx}^2}{r^{2n+2}-r^{2n}}
\le 
\frac{\norm{T^{n+2}x}^2-\norm{T^{n+1}x}^2}{r^{2n+4}-r^{2n+2}}.  
\]
If we consider $x\in\H$ fixed, this means that 
the broken line with vertices in $(r^{2n},\norm{T^nx}^2)_{n\in\Z}$ is concave. 
Then, for any its edge, the straight line containing it is above the whole broken line.  
Since $x$-coordinates $r^{2n}$ of the vertices go to $\infty$ as $n\to-\infty$ and 
the broken line is above the $x$-axis, all edges have non-negative slope. 
In particular, $\norm{Tx}^2\le \norm{x}^2$, so $T$ is a contraction.
The second part of the statement follows using Proposition~\ref{prop:bothCalpha}. 
\end{proof}

As usual, let $\sigma(T)$ denote the spectrum of $T$. 
The following result is an immediate consequence of Proposition~\ref{prop:Tcontraction}. 

\begin{corollary}
If $T$ is in $\BC_\al$, then $\si(T)\ss \cl\A$, and the limits
\[
L^+(T,x) := \lim_{n \to \infty} \norm{T^n x}^2, \quad
L^-(T,x) := \lim_{n \to \infty} \norm{r^n T^{-n} x}^2 
\]
exist for all $x$ in $\H$.
\end{corollary}

Notice that $L^+(T,Tx) = L^+(T,x)$ and $L^-(T,Tx) = r^2 L^-(T,x)$ for all $x$ in $\H.$ 
This obvious fact will be implicitly used below.

\begin{lemma}\label{lem:normPos}
If $T$ is in $\BC_\al$, then
\begin{equation}\label{eq:seriesDefects}
\sum_{n=0}^{\infty} \norm{DT^n x}^2 = \norm{Tx}^2 - r^2 \norm{x}^2 + (r^2-1) L^+(T,x)
\end{equation}
for all $x$ in $\H$. 
In particular, the series on the left hand side converges. 
\end{lemma}

\begin{proof}
Note that
\[
\norm{DT^n x}^2 = -\norm{T^{n+2}x}^2 + (1+r^2) \norm{T^{n+1}x}^2 -r^2\norm{T^n x}^2. 
\]
for all $x \in\H$ and all $n\in\Z$. Therefore
\[
\begin{split}
\sum_{n=0}^{N} \norm{DT^n x}^2 &= -\sum_{n=2}^{N+2} \norm{T^n x}^2 
+ (1+r^2) \sum_{n=1}^{N+1} \norm{T^n x}^2 - r^2\sum_{n=0}^{N} \norm{T^n x}^2 \\
&= -r^2 \norm{x}^2 + \norm{Tx}^2 + r^2 \norm{T^{N+1}x}^2 - \norm{T^{N+2}x}^2.
\end{split}
\]
for all $N\in\N$. 
The statement follows letting $N\to\infty$. 
\end{proof}

\begin{lemma}\label{lem:normNeg}
If $T$ is in $\BC_\al$, then
\begin{equation}\label{eq:seriesDefectsNeg}
\sum_{n=-\infty}^{-1} r^{-2n-2} \norm{DT^n x}^2
=
- \norm{Tx}^2 + \norm{x}^2 + (r^2-1) L^-(T,x)
\end{equation}
for all $x$ in $\H$. 
In particular, the series on the left hand side converges. 
\end{lemma}

\begin{proof}
By Proposition~\ref{prop:bothCalpha}, 
$rT^{-1}\in\BC_\al$. 
Let $\wt{D}$ denote the defect operator of $rT^{-1}$. 
A straightforward computation gives that
\begin{equation}\label{eq:DtildeD}
r^{-2} \norm{\wt{D}T^2x}^2 = \norm{Dx}^2 
\end{equation} 
for all $x \in\H$. 
Using Lemma~\ref{lem:normPos} for $rT^{-1}$, we obtain
\begin{equation}\label{eq:seriesDefectsPos}
\sum_{n=0}^{\infty} \norm{\wt{D} r^n T^{-n} y}^2 
= \norm{rT^{-1}y}^2 - r^2 \norm{y} + (r^2-1) L^+(rT^{-1},y)
\end{equation}
for all $y \in\H$. Set $y = r^{-1}Tx$. Then
\[
L^+(rT^{-1},y) 
= \lim_{n \to \infty} \norm{r^n T^{-n} y}^2 
= \lim_{n \to \infty} \norm{r^{n-1} T^{-n+1} x}^2 
= L^-(T,x).
\]
Hence the right hand side of \eqref{eq:seriesDefectsPos} 
is equal to 
the right hand side of \eqref{eq:seriesDefectsNeg}. 
Using \eqref{eq:DtildeD} we 
get 
\[
\sum_{n=0}^{\infty} \norm{\wt{D} r^n T^{-n} y}^2 
= \sum_{m=-\infty}^{-1} r^{-2m-2} \norm{DT^m x}^2.
\]
Therefore \eqref{eq:seriesDefectsNeg} follows. 
\end{proof}

\begin{theorem}\label{thm:norm}
If $T$ is in $\BC_\al$, then 
\[
\begin{split}
\norm{x}^2 = \,
&\dfrac{1}{1-r^2} \sum_{n=0}^{\infty} \norm{DT^n x}^2 + 
\dfrac{1}{1-r^2} \sum_{n=-\infty}^{-1} r^{-2n-2} \norm{DT^n x}^2 \\
&+L^+(T,x) + L^-(T,x)
\end{split}
\]
for all $x$ in $\H$. 
In particular, both series on the right hand side converge. 
\end{theorem}

\begin{proof}
Add up equations \eqref{eq:seriesDefects} and \eqref{eq:seriesDefectsNeg}, 
and rearrange the terms. 
\end{proof}

Now we are headed to obtain the lifting of the model for operators in $\BC_\al$. 
Theorem~\ref{thm:norm} foreshadows its structure. 
Let $T\in\BC_\al$. Consider the operator 
\[
\CO_{T,D}\colon H\to H^2_r(\B,\FD), \quad 
\CO_{T,D}x(z)=D(z-T)^{-1}x, \quad x\in H, z\in\B. 
\]
The function $\CO_{T,D}x(z)$ can be represented as 
\[
-\sum_{n=0}^{\infty} (DT^{-n-1}x)z^n \ \text{ for } \lvert z\rvert<r, 
\qquad
\sum_{n=-\infty}^{-1} (DT^{-n-1}x)z^n \ \text{ for } \lvert z\rvert>1.   
\]
Hence 
\[
\begin{split}
\lVert \CO_{T,D}x\rVert_{H^2_r(\B,\FD)}^2 
&= \frac{1}{1-r^2}\sum_{n=0}^{\infty} r^{2n}\lVert DT^{-n-1}x\rVert^2
+\frac{1}{1-r^2}\sum_{n=-\infty}^{-1} \lVert DT^{-n-1}x\rVert^2\\
&= \frac{1}{1-r^2}\sum_{n=0}^{\infty} \lVert DT^{n}x\rVert^2
+\frac{1}{1-r^2}\sum_{n=-\infty}^{-1} r^{-2n-2}\lVert DT^{n}x\rVert^2,
\end{split}
\]
for all $x$ in $\H$. Therefore, by Theorem~\ref{thm:norm}, $\CO_{T,D}$ is a contraction.

We construct the second component of the lifting using a well-known 
argument (cf.\ \cite{Ker89,SzN47}), that we also employed in \cite[Section~4]{BY}. 
By Proposition~\ref{prop:Tcontraction} and the polarization identity,  
one can define on $\H$ the sesquilinear form 
\[
\left\langle x,y \right\rangle_+ 
:= \lim_{n \to \infty} \left\langle T^nx, T^ny \right\rangle. 
\]
Let $\H_0$ be the subspace of vectors $x$ in $\H$ such that 
$L^+(T,x) = \left\langle x,x \right\rangle_+ =0$. 
Denote by $\wt{H}_+$ the Hilbert space obtained as the completion of 
the quotient space $\H/\H_0$ in the 
pre-Hilbert norm $x\mapsto \langle x,x \rangle_+^{1/2}$.  
Let $W_+:H \to \wt{H}_+$ be the operator that maps each vector $x$ to its class $[x]_+$. 
In the same way, thanks to Proposition~\ref{prop:Tcontraction}, 
starting from the sesquilinear form 
\[
\left\langle x,y \right\rangle_- 
:= \lim_{n \to \infty} \left\langle r^nT^{-n}x, r^nT^{-n}y \right\rangle 
\]
on $\H$, we define the Hilbert space $\wt{H}_-$ and the operator $W_-:H \to \wt{H}_-$ 
mapping  each vector $x$ to its class $[x]_-$. 
Denote by $W$ the operator $(W_+,W_-)$, and let $\wt{H}$ be the closure of its range. 
That is, $\wt{H}$ is a closed subspace of $\wt{H}_+ \oplus \wt{H}_-$, and 
\[
W : H \to \wt{H}, \quad Wx = W_+x \oplus W_-x. 
\]
Notice that 
\[
\norm{Wx}^2=\norm{W_+x}^2+\norm{W_-x}^2=L^+(T,x)+L^-(T,x)
\]
for all $x\in\H$. Therefore, as a consequence of Theorem~\ref{thm:norm} we obtain 
our lifting theorem. 

\begin{theorem}[Lifting of the model]\label{thm:lifting}
If $T$ is in $\BC_\al$, the transformation 
\[
(\CO_{T,D},W)\colon\H\to H^2_{r}(\B,\FD) \oplus\wt{H}, \quad 
(\CO_{T,D},W)(x)=\CO_{T,D}x\oplus Wx
\]
is an isometry. 
\end{theorem}

In order to obtain the model theorem for operators in $\BC_\al$, 
we first need to analyze the particular case of operators $T$ such that $\al(T^*,T) = 0$. 

\begin{theorem}\label{thm:main2}
Suppose that $T$ is invertible.  
The following statements are equivalent. 
\begin{enumerate}[\rm(i)]
\item 
$\al(T^*,T) = 0$.
\item 
$T$ is a subnormal operator whose minimal normal 
extension can be written as a sum $U_+ \oplus U_-$, 
where $U_+$ and $r^{-1}U_-$ are unitary operators. 
\end{enumerate}
In this case, either $T$ is normal or $\sigma(T)=\cl{\A}$.
\end{theorem}

\begin{proof}
Assume that $\al(T^*,T) = 0$. In particular $T \in \BC_\al$.  
Note that $W_+\colon\H\to\wt{H}_+$ is onto. 
Since $W_+x=W_+Tx$ for all $x\in\H$, the formula $U_+W_+x=W_+Tx$ 
defines an isometric operator $U_+$ on $\wt{H}_+$. 
Since the range of $U_+$ is $\wt{H}_+$, $U_+$ is unitary. 
Analogously, the operator $U_-$ on $\wt{H}_-$ given by $U_-W_-x=W_-Tx$ 
satisfies that $r^{-1}U_-$ is a unitary operator. 
Set $U:=U_+ \oplus U_-$, which acts on $\wt{H}_+ \oplus \wt{H}_-$. 
Then $U$ is a normal operator with 
\[
\sigma(U)\ss\{ \abs{z} = r \}\cup\{ \abs{z} = 1 \}. 
\]
Since $\FD={0}$, Theorem~\ref{thm:lifting} gives that 
$W\colon\H\to\wt{H}$ is a unitary operator. 
Therefore $T$ is unitarily equivalent to the restriction of $U$ to $\wt{H}$, 
which has the desirable properties of (ii).

Now suppose that $T$ satisfies (ii). 
Let $U=U_+ \oplus U_-$ be its minimal unitary extension.  
Then $U$ is bounded and $\sigma(U)$ is contained in 
the union of the circles $\{\abs{z}=r\}$ and $\{\abs{z}=1\}$. 
Write each vector $x$ in the space where $U$ acts as a pair $(x_+,x_-)$, 
so that $Ux=U_+x_+ + U_-x_-$. 
Then 
\[
\norm{Ux}^2 = \norm{U_+x_+}^2 + \norm{U_-x_-}^2 = \norm{x_+}^2 + r^2\norm{x_-}^2. 
\]
Hence, a straightforward computation shows that 
\[
(1+r)^2 \norm{Ux}^2 - \norm{U^2x}^2 - r^2 \norm{x}^2 = 0,
\]
so $\al(U^*,U) = 0$. Since $T$ is a part of $U$, (i) follows. 

The last statement follows from \cite[Theorem~II.2.11~(c)]{Conwaybook}.
\end{proof}

The last result we need to prove Theorem~\ref{thm:main-b} is the following. 

\begin{proposition}
\label{prop-C-alpha}
For any Hilbert space $\E,$ the operator 
\[
M_z^t\colon H^2_{r}(\B,\E))\to H^2_{r}(\B,\E))
\]
given by \eqref{eq:Mzt} belongs to $\BC_\al$. 
\end{proposition}

\begin{proof}
Identifying $f\in H^2_{r}(\B,\E)$ with the sequence $\{f_n\}_{n\in\Z}$ 
of its coefficients, a straightforward computation gives that 
\[
(1+r^2)\norm{M_z^tf}^2-\norm{(M_z^t)^2f}^2-r^2\norm{f}^2=\norm{f_{0}}^2\ge0. 
\]
Hence $\al((M_z^t)^*,M_z^t)\ge0$. If we view 
$M_z^t$ as a bilateral shift, we get that it is invertible. 
\end{proof}

\begin{proof}[Proof of Theorem~\ref{thm:main-b}]
Assume that $\al(T^*,T) \ge 0$, that is, $T\in\BC_\al$. 
Then we can obtain a normal operator $U=U_+\oplus U_-$ 
as at the beginning of the proof of Theorem~\ref{thm:main2}. 
Let $S$ be the restriction of $U$ to $\wt\H$. Note that $S$ has the desired properties 
of (ii). Moreover, $SW=WT$. Since $M_z^t$, acting on $H^2_{r}(\B,\FD)$, 
satisfies $M_z^t\CO_{T,D}=\CO_{T,D}T$, using Theorem~\ref{thm:lifting} 
we obtain that $T$ is unitarily equivalent to a part of $M_z^t \oplus S$. 

Now assume (ii). Let $U$ be the normal minimal extension of $S$. 
Since $\sigma(U)$ is contained in the union of the circles 
$\{\abs{z}=r\}$ and $\{\abs{z}=1\}$, $U$ can be written as 
an orthogonal sum $U_+\oplus U_-$, where $U_+$ and $r^{-1}U_-$ are unitaries. 
By Theorem~\ref{thm:main2}, $\al(U^*,U)=0$, so we also have $\al(S^*,S)=0$. 
Since $M_z^t$ is in $\BC_\al$ (see Proposition~\ref{prop-C-alpha}) 
and $T$ is a part of $M_z^t \oplus S$, we obtain (i). 
\end{proof}

Notice that Theorem~\ref{thm:main2} can be seen as a particular case of 
Theorem~\ref{thm:main-b}, when the structure involving the defect operator disappears. 

\section{Dual models}
\label{sec:duals}

In this section we discuss the duality behind the models involving operators 
$M_z$ acting on spaces of functions on $\A$, on one hand, 
and the models in terms of operators $M_z^t$, acting on spaces of functions on $\B$, 
on the other hand.
In \cite{BYdoms}, we exploit this duality in the context of multiply connected domains. 

\begin{proposition}\label{prop:duality-r}
The operator $M_z^t$ acting on $H^2_r(\B,E)$ is dual to the operator 
$M_z$ acting on $H^2_{r}(\A,\E)$ via the duality 
\[
\langle f, g\rangle = 
\int_{\T} \langle f(z), g(\bar z)\rangle_E\, dz+\int_{r\T} \langle f(z), g(\bar z)\rangle_E\, dz,  
\]
where $f$ is in $H^2_r(\B,E)$ and $g$ is in $H^2_{r}(\A,\E)$. 
\end{proposition}

\begin{proof}
It is easy to see that $M_z^t$ acting on $H^2_r(\B,\E)$ is unitarily equivalent to 
the backward shift 
\[
B\colon\ell^2_r(\E)\to\ell^2_r(\E), \quad B(\{f_n\}_{n\in\Z})=\{f_{n+1}\}_{n\in\Z}, 
\]
where $\ell^2_r(\E)$ is the space of sequences $\{f_n\}_{n\in\Z}$ in $\E$ with finite norm 
\[
\norm{\{f_n\}_{n\in\Z}}_{\ell^2_r(\E)}:=\frac{1}{1-r^2}\left(
\sum_{n=0}^{\infty}\norm{f_n}^2+\sum_{n=-\infty}^{-1}r^{-2n-2}\norm{f_n}^2
\right). 
\]
In the same way, $M_z$ acting on $H^2_{r}(\A,\E)$ is unitarily equivalent to 
the forward shift 
\[
F\colon\ell^2_{r,-}(\E)\to\ell^2_{r,-}(\E), \quad F(\{g_n\}_{n\in\Z})=\{g_{n-1}\}_{n\in\Z}, 
\]
where $\ell^2_{r,-}(\E)$ is the space of sequences $\{g_n\}_{n\in\Z}$ in $\E$ with finite norm 
\[
\norm{\{g_n\}_{n\in\Z}}_{\ell^2_{r,-}(\E)}:=(1-r^2)\left(
\sum_{n=0}^{\infty} \frac{1}{r^{2n}}\norm{g_n}^2+\sum_{n=-\infty}^{-1} \norm{g_n}^2
\right). 
\]
Finally, the duality can be written as 
\[
\ip{\{f_n\}}{\{g_n\}}=\sum_{n\in\Z}\ip{f_n}{g_{-n-1}}_\E, 
\]
where $\{f_n\}_{n\in\Z}\in\ell^2_r(\E)$ and $\{g_n\}_{n\in\Z}\in\ell^2_{r,-}(\E)$. 
Now an easy computation shows that the statement holds. 
\end{proof}

Therefore, as an immediate consequence of Theorem~\ref{thm:main-b} 
and Proposition~\ref{prop:duality-r}, we obtain a dual model for operators 
in the class $\BC_\al$. 

\begin{theorem}[Dual model theorem]\label{thm-1-2}
Suppose that $T$ is invertible. 
Then $T$ is in $\BC_\al$ if and only if $T^*$ is a 
compression of an operator of the form 
\[
(M_z \text{ on } H^2_{r}(\A,\E))\oplus U_0\oplus rU_1,
\] 
where $\E$ is a Hilbert space, and $U_0$ and $U_1$ are unitary operators.
\end{theorem}

In what follows, we will use the Hilbert 
spaces $H^2(\A, \E, \w)$, defined in~\eqref{eq:norm-omega}, 
in the scalar-valued case when  $\E=\C$. Putting $\om=a>0$, 
we get the Hilbert norms  
\[
\norm{f}_{H^2(\A,\C,a)}^2=\int_\T\abs{f(z)}^2\,\abs{dz}
+a\int_{r\T}\abs{f(z)}^2\,\abs{dz}. 
\]
Notice that $\{z^n\}_{n\in\Z}$ is an orthogonal basis in $H^2(\A,\C,a)$ with 
\[
\norm{z^n}_{H^2(\A,\C,a)}^2=1+a^2r^{2n}
\] 
for all $n\in\Z$. 
Therefore, in terms of the coefficients $\{f_n\}_{n\in\Z}$ of the Laurent series 
of $f$ in $\A$ (that is, $f(z)=\sum_{n\in\Z}f_nz^n$), we have 
\[
\norm{f}_{H^2(\A,\C,a)}^2=\sum_{n=-\infty}^{\infty}(1+a^2r^{2n})\abs{f_n}^2. 
\]
Now consider the space $H^2(\B,\C,a)$ of all functions $f\colon\B\to\C$ 
given by 
\[
\sum_{n=0}^{\infty} f_nz^n \ \text{ for } \lvert z\rvert<r, 
\qquad
\sum_{n=-\infty}^{-1} f_nz^n \ \text{ for } \lvert z\rvert>1,
\]
equipped with the Hilbert space norm 
\[
\norm{f}_{H^2(\B,\C,a)}^2 := \sum_{n=-\infty}^{\infty} \frac{1}{1+a^2r^{-2n-2}}\, \abs{f_n}^2. 
\]
For any Hilbert space $\E$, set 
\[
H^2(\B,\E,a):=H^2(\B,\C,a)\otimes\E, \quad H^2(\A,\E,a):=H^2(\A,\C,a)\otimes\E. 
\]

\begin{proposition}\label{prop:duality-a}
The operator $M_z^t$ acting on $H^2(\B,E,a)$ is dual to the operator 
$M_z$ acting on $H^2(\A,\E,a)$. 
\end{proposition}

\begin{proof}
The argument is essentially the same as in the proof of Proposition~\ref{prop:duality-r}. 
Indeed, the same duality works. Hence we omit it. 
\end{proof}

\begin{lemma}\label{lem:MztinCa}
The operator $M_z^t$ acting on $H^2(\B,E,a)$ is in $\BC_\al$. 
\end{lemma}

\begin{proof}
Identifying functions $f$ in $H^2(\B,E,a)$ with the sequence $\{f_n\}_{n\in\Z}$ of 
its coefficients, it is immediate that $M_z^t$ can be identified with the forward shift 
$F$ acting on the weighted space $\ell_\om^2(\Z)$ of bilateral sequences with  
\[
\om_n=\frac{1}{1+a^2r^{-2n-2}}. 
\]
Note that $F$ satisfies \eqref{eq:normDx} if 
\[
\frac{1+r^2}{1+ar^{-2n-4}}-\frac 1 {1+ar^{-2n-6}}- \frac{r^2}{1+ar^{-2n-2}}
\ge 0 
\]
for all $n\in \Z$. Setting $r^2=\rho$ and $ar^{-2n-6}=x$, this follows from 
\[
\frac{1+\rho}{1+x\rho}
-\frac 1 {1+x}
- \frac{\rho}{1+x\rho^2}
=\frac{(\rho-1)^2(\rho+1)}{(1+x)(1+x\rho)(1+x\rho^2)}
\ge 0.
\]
Hence $F$ is in $\BC_\al$, as we wanted to prove. 
\end{proof}

We remark that a special role of spaces 
$H^2(\B,E,a)$ for the values $a=r^{2m}$, $m\in \mathbb{Z}$, 
was observed in \cite{McCS12} (see Proposition ~2.2 of that work). 
Namely, for these (and only these) values, 
the corresponding reproducing kernel $k(z,w)$ does not vanish on $\A\times \A$. 
The commutant lifting theorem given in \cite{McCS12} 
involves the operator class, defined in terms of 
the operator $M_z$ on $H^2(\B,E,a)$, where $a$ takes one of these special values.

\begin{proof}[Proof of Theorem~\ref{thm:strictInclusions}]
First, notice that the inclusion $\BC_\al\ss\BC_{1,r}$ has already been proved in 
Proposition~\ref{prop:Tcontraction}. 
Now take $T\in\Sp(\A)$ and let us see that $T\in\BC_\al$. 
Since $\cl\A$ is also $1$-spectral for $T^*$, by Theorem~A we know that  
$T^*$ is unitarily equivalent to a compression of the operator 
\[
(M_z \text{ on }H^2(\A,\E, \om))\oplus U_1\oplus rU_2
\] 
to its coinvariant 
subspace, where $\om\in L(\E)$ is a weight and $U_1$, $U_2$ are unitaries. 
Therefore $T$ is unitarily equivalent to a part of the operator 
\[
\big(M_z \text{ on }H^2(\A,\E,\om)\big)^*\oplus U_1^*\oplus rU_2^*.
\] 
Hence, it suffices to check that any operator of this form is in $\BC_\al$. 
It is clear that the operator $N=U_1^*\oplus rU_2^*$ is in $\BC_\al$ (indeed $\al(N^*, N)=0$).  
Let us see now that the operator $M_z$ on $H^2(\A,\E,\om)$ is also in $\BC_\al$. 
By the spectral theorem, $\om\in L(\E)$ is unitarily equivalent to the multiplication operator 
$M_a f(a)= a f(a)$, acting on a direct integral of Hilbert spaces 
\[
\int^\oplus H(a)\, d\mu(a), 
\]
where $\mu$ is a finite Borel measure, concentrated on $\si(\om)$. 
In this representation, the operator $M_z$  on $H^2(\A,\E,\om)$ rewrites as 
\[
\int^\oplus M_z \, d\mu(a)\quad \text{on }\int^\oplus H^2\big(\A,H(a), aI_{H(a)}\big) \, d\mu(a). 
\]
Therefore, it is enough to check only the scalar case. 
That is, we need to prove that the adjoint to $M_z$ on the scalar space 
$H^2(\A,\C, a)$ is in $\BC_\al$ for all $a>0$, which follows from 
Proposition~\ref{prop:duality-a} and Lemma~\ref{lem:MztinCa}. 
Hence, we have proved the inclusion $\Sp\A\ss\BC_\al$. 

Next, we show that the inclusion $\BC_\al\ss\BC_{1,r}$ is strict for all $0<r<1$. 
Consider the matrices 
\[
T_1:=
\left(
\begin{matrix}
1/\sqrt{2}\,  &  \,0 \\
1/2 \,& \, 1/\sqrt{2}
\end{matrix}
\right), 
\quad 
T_2:=
\left(
\begin{matrix}
\sqrt{r}\,  &  \,0 \\
1-r \,& \, \sqrt{r}
\end{matrix}
\right).  
\]
It is easy to check that $\frac 12 \norm{x}\le\norm{T_1 x}\le \norm{x}$ and 
$r\norm{x}\le\norm{T_2x}\le \norm{x}$ for all $x\in H$.  
Now we test the left hand side of \eqref{eq:normDx} with $x=(1,0)$ for both $T_1$ and $T_2$. 
For $T_1$ we obtain the result $-r^2/4$. For $T_2$ we get $(r-1)^2(r^2-3r+1)$. 
Therefore, if $0<r\le1/2$ then $T_1$ is in $\BC_{1,r}$ but not in $\BC_\al$, 
while if $1/2<r<1$ then $T_2$ is in $\BC_{1,r}$ but not in $\BC_\al$. 

Finally, let us see that the inclusion $\Sp\A\ss\BC_\al$ is strict. 
Consider the shifts $B$ and $F$ given in the proof of Proposition~\ref{prop:duality-r}. 
There we established the duality between these two operators. 
Testing the left hand side of \eqref{eq:normDx} for the operator $F$ and 
the vector $f=\{f_n\}_{n\in\Z}$, with $f_0=1$ and $f_n=0$ if $n\neq0$, we have 
\[
\begin{split}
(1+r^2)\norm{Ff}^2-\norm{F^2f}^2-r^2\norm{f}^2
&=
(1-r^2)\left(\frac{1+r^2}{r^2}-\frac{1}{r^4}-r^2\right)\\
&=
\frac{(1-r^2)^2}{r^4}(r^4-1),
\end{split}
\]
which is negative for all $0<r<1$. 
Hence, $B\in\BC_\al$ (see Proposition~\ref{prop-C-alpha}) but $F\notin\BC_\al$. 
Since the class $\Sp\A$ is invariant by passing to the adjoint operator 
(i.e., $T\in\Sp\A$ if and only if $T^*\in\Sp\A$), we obtain that $B\in\BC_\al\sm\Sp\A$. 
\end{proof} 

\begin{remark}
In an earlier version of this paper (published in arxiv), 
we asked whether it is true that $\Sp\A=\{T\in \BC_\al: T^*\in \BC_\al\}$. 
Recently G.~Tsikalas gave a concrete example, showing that 
$\Sp\A$ is strictly smaller than the set 
$\{T\in \BC_\al: T^*\in \BC_\al\}$. 
\end{remark}

\section{Consequences of the model}
\label{sec:consequences}

In this section we present some results derived from the explicit model 
obtained in Theorem~\ref{thm:main-b} for operators in $\BC_\al$. 
For instance, we give a concrete value $K$ such that 
$\cl\A$ is completely $K$-spectral for all operators in $\BC_\al$, 
and we establish a characterization of the inclusion of classes $\BC_\al\ss\BC_\be$. 

\begin{proposition}\label{prop:inclusionmapJa}
The inclusion map 
\[
J_a\colon H^2_r(\B,\E)\hookrightarrow H^2(\B,\E,a), \quad J_af=f
\] 
is well-defined. Moreover, it is a bijection and 
\[
\norm{J_a}^2=\frac{1-r^2}{\min\{a^2r^{-2},1\}}, 
\quad 
\norm{J_a^{-1}}^2=\frac{1+a^2r^{-2}}{1-r^2}.
\]
\end{proposition}

\begin{proof}
For the orthogonal basis $\{z^n\}_{n\in\Z}$ we have
\[
\lVert z^n\rVert^2_{H^2(\B,\E,a)}=\frac{1}{1+a^2r^{-2n-2}}
\]
and 
\[
\lVert z^n\rVert^2_{H^2_r(\B,\E)}=
\left\{
\begin{array}{ll}
r^{2n}/(1-r^2) & \text{ if } n\ge 0\\
1/(1-r^2) & \text{ if } n\le -1
\end{array}
\right.. 
\]
Therefore 
\begin{equation}\label{eq:Ja}
\lVert J_{a}\rVert^2=\max
\left\{
\sup_{n\ge0}\frac{1-r^2}{(1+a^2r^{-2n-2})r^{2n}}, 
\sup_{n\le-1}\frac{1-r^2}{1+a^2r^{-2n-2}}
\right\}, 
\end{equation}
and 
\begin{equation}\label{eq:Jainverse}
\lVert J_{a}^{-1}\rVert^2=\max
\left\{
\sup_{n\ge0}\frac{(1+a^2r^{-2n-2})r^{2n}}{1-r^2}, 
\sup_{n\le-1}\frac{1+a^2r^{-2n-2}}{1-r^2}
\right\}. 
\end{equation}
Recall that $0<r<1$. 
Hence, in \eqref{eq:Ja} we need to compare the cases $n\to\infty$ and $n\to-\infty$, 
while in \eqref{eq:Jainverse} the maximum is reached for $n=0$. 
\end{proof}

\begin{theorem}\label{thm:sqrt2}
$\cl{\A}$ is completely $\sqrt{2}$-spectral for all operators in $\BC_\al$. 
\end{theorem}

\begin{proof}
Let $T\in\BC_\al$. 
Given $a>0$, the operator $M_z$ acting on $H^2(\A,\FD,a)$ 
has $\cl{\A}$ as a complete $1$-spectral set (see Theorem~A). 
Therefore, its adjoint operator, $M_z^t$ acting on $H^2(\B,\FD,a)$, also 
has $\cl{\A}$ as a complete $1$-spectral set. 
By Theorem~\ref{thm:main-b} and Proposition~\ref{prop:inclusionmapJa}, 
it follows that $\cl{\A}$ is a complete $K_{a}$-spectral set for $T$, where 
$K_{a}:=\lVert J_{a}\rVert\cdot\lVert J_{a}^{-1}\rVert$. 
Using Proposition~\ref{prop:inclusionmapJa} again, we have 
$\inf_{a>0} K_{a} = \sqrt{2}$. This infimum is indeed attained when $a=r$. 
Now the statement follows.
\end{proof}

In the recent preprint \cite{Tsi21}, 
Tsikalas also obtained the constant $\sqrt{2}$ of Theorem~\ref{thm:sqrt2}, 
and proved that, in fact, this constant is the best possible. 
That work contains an alternative proof of Proposition~\ref{prop:Tcontraction} 
(see \cite[Lemma 4.1]{Tsi21}). 
In his other recent preprint \cite{Tsi21_A}, 
he showed that the corresponding constant for the class $\BC_{1,r}$ is at least $2$, 
which, in particular, provides an alternative proof of the fact that 
$\BC_\alpha$ is strictly smaller than $\BC_{1,r}$. 

\begin{theorem}\label{thm:inclusionCalphaCbeta}
Let $\be(t)=(1-t)(t-s^2)$ for some $0<s<1$. 
Then $\BC_\al\ss\BC_\be$ if and only if $s\le r$. 
\end{theorem}

\begin{proof}
Let $U$ be a unitary operator. 
Using \eqref{eq:normDx}, note that $rU$ is in $\BC_\be$ if and only if 
\[
(1-r^2)(r^2-s^2)=(1+s^2)r^2-r^4-s^2\ge0. 
\]
Hence, $\BC_\al\ss\BC_\be$ implies $s\le r$. 
Now suppose that $s\le r$. Using the model theorem for operators in $\BC_\al$, 
it remains to prove that $M_z^t$ acting on the space $H_r^2(\B,\E)$ is in $\BC_\be$, 
for any Hilbert space $\E$. Equivalently, we want to prove that the backward shift $B$ 
acting on $\ell_r^2(\E)$ is in $\BC_\be$. For any sequence $f=\{f_n\}_{n\in\Z}$ we have 
\[
(1+s)^2\norm{Bf}^2-\norm{B^2f}^2-s^2\norm{f}^2=
\norm{f_0}^2+(r^2-s^2)\sum_{n=-\infty}^{-1}r^{-2n-2}\norm{f_n}^2,
\]
which clearly is non-negative. 
Therefore the statement follows using \eqref{eq:normDx} again. 
\end{proof}

\section*{Acknowledgments}

The first author was supported by 
National Science Centre, Poland grant UMO-2016/21/B/ST1/00241.
The second author acknowledges partial support by
Spanish Ministry of Science, Innovation and
Universities (grant no. PGC2018-099124-B-I00) and
the ICMAT Severo Ochoa project SEV-2015-0554 of the Spanish Ministry of Economy and
Competitiveness of Spain and the European Regional Development
Fund, through the ``Severo Ochoa Programme for Centres of Excellence
in R$\&$D''.
The second author also acknowledges
financial support from the Spanish Ministry of 
Science and Innovation, through the ``Severo Ochoa Programme for 
Centres of Excellence in R$\&$D'' (SEV-2015-0554) and from the Spanish 
National Research Council, through the ``Ayuda extraordinaria a 
Centros de Excelencia Severo Ochoa'' (20205CEX001).

\bibliographystyle{siam}
\bibliography{biblio_Annulus}

\end{document}